%% file: group-action.tex
\documentclass[reqno]{gokova}
\usepackage{epsfig,graphicx}

\begin{document}
\input goksty.tex      

\setcounter{page}{1}
\volume{13}

\title[]{Group actions on $4$-manifolds: some recent results and open questions}
\author[]{Weimin Chen}

\thanks{The author was supported in part by NSF grant DMS-0603932}

\address{University of Massachusetts, Amherst}
\email{wchen@math.umass.edu}

\address{}
\email{}

\begin{abstract}
A survey of finite group actions on symplectic $4$-manifolds is given with a special emphasis on
results and questions concerning smooth or symplectic classification of group actions, group actions 
and exotic smooth structures, and homological rigidity and boundedness of group actions. 
We also take this opportunity to include several results and questions which did not appear elsewhere.
\end{abstract}
\maketitle
\tableofcontents
\section{Introduction}
Let $X$ be a compact closed, oriented, simply connected, topological $4$-manifold, and let $G$ be
a finite group. Consider continuous actions of $G$ on $X$ which are orientation preserving, i.e., for 
every $g\in G$, the homeomorphism $g:X\rightarrow X$ is orientation preserving. We shall further 
assume that the actions of $G$ are locally linear, which means that for any $x\in X$ with 
nontrivial isotropy subgroup $G_x:=\{g\in G| g\cdot x=x\}\neq \{1\}$, the action of $G_x$ near $x$ is
modeled by a linear action, given by a faithful representation $\rho_x: G_x\rightarrow GL(4,\R)$. 
Note that $\rho_x$ is uniquely determined up to linear conjugation because nonlinear similarity 
begins in dimension $6$ (cf. \cite{CSSW}). For such an action of $G$ on $X$, one can associate to 
it the following data
$$
\Gamma:= (\rho, X^G, \{\rho_x\}),
$$
where $\rho:G\rightarrow \text{Aut } (H^2(X),\cup)$ is the induced action on the second cohomology 
(note that it preserves the cup product $\cup$ on $H^2(X)$), $X^G:=\{x\in X|G_x\neq \{1\}\}$ is the 
singular set of the action, and $\rho_x: G_x\rightarrow GL(4,\R)$ is the linear representation modeling 
the group action near $x$, $\forall x\in X^G$. The action is called {\it free} if $X^G=\emptyset$, {\it
semifree} if $G_x=G$ for all $x\in X^G$, and {\it pseudofree} if $X^G$ is a finite set. 
A good starting point for understanding the action of $G$
on $X$ is the associated data $\Gamma:= (\rho, X^G, \{\rho_x\})$. 

The fundamental case to consider is when $G\equiv\Z_p$ is a cyclic group of prime order $p>1$. In this
case $X^G$ is simply the set of points fixed under $G$, which is a disjoint union of isolated points 
and $2$-dimensional surfaces (because the action of $G$ preserves the orientation of $X$). For each 
$x\in X^G$, $G_x=G$ (i.e., the action is semifree), and moreover, if $x=m$ is an isolated 
fixed point, $\rho_x$ is given by an unordered pair of nonzero integers $(a_m,b_m)$ unique up to a simultaneous change of sign and congruence modulo $p$, and if $x\in Y$ lies in a fixed surface 
$Y$, $\rho_x$ is given by a nonzero integer $c_Y$, unique up to a change of sign and congruence 
modulo $p$. More precisely,
in the former case $\rho_x: \lambda\mapsto diag (\lambda^{a_m},\lambda^{b_m})$ and in the latter
case $\rho_x:\lambda\mapsto diag (1,\lambda^{c_Y})$, where $\lambda$ is a $p$-th root of unity. 
Finally, for $\rho:G\rightarrow \text{Aut } (H^2(X),\cup)$, the corresponding integral $\Z_p$-representation 
on $H^2(X)$ can be decomposed into a direct sum 
$$
H^2(X)=\Z[\Z_p]^r\oplus \Z^t\oplus \Z[\mu_p]^s
$$
for some integers $r,t,s\geq 0$, where the group ring $\Z[\Z_p]$ is the regular representation of 
$\Z$-rank $p$, $\Z$ is the trivial representation of $\Z$-rank $1$, and $\Z[\mu_p]$ is the representation 
of cyclotomic type of $\Z$-rank $(p-1)$, which is the kernel of the augmentation homomorphism $\Z[\Z_p]\rightarrow \Z$ (see \cite{KS}). (Here $\mu_p\equiv \exp(2\pi i/p)$.)

We collect below the main constraints that have to be satisfied by 
$\Gamma:= (\rho, X^G, \{\rho_x\})$ where $G\equiv\Z_p$. See \cite{CK1}, Section 3,  for a more
complete review. 

\begin{prop}
(1) $b_2(X)=rp+t+s(p-1)$.

(2) (Lefschetz fixed point formula) $\chi(X^G)=t-s+2$.

(3) (cf. \cite{Ed1})  If $X^G\neq \emptyset$, then $b_1(X^G;\Z_p)=s$.  In particular, there are no summands
of cyclotomic type in $H^2(X)$ (i.e., $s=0$) if the action is pseudofree with nonempty fixed point set. 

(4) (cf. \cite{Ed1}) Suppose $G=\langle g\rangle=\Z_2$. If $X^G$ is empty or a finite set, then 
$g^\ast \alpha\cup\alpha=0 \pmod{2}$ for all $\alpha\in H^2(X)$. 
\end{prop}

\begin{prop}
(1) (G-signature Theorem, \cite{HZ, Gor}) 
Assume each $2$-dimensional component $Y\subset X^G$ is orientable. 
Set $\text{Sign}(g,X)=tr(g)|_{H^{2,+}(X;\R)}-tr(g)|_{H^{2,-}(X;\R)}$, $\forall g\in G$. Then 
$$
\text{Sign}(g,X)=\sum_{m\in X^G} -\cot(\frac{a_m\pi}{p})\cdot 
                      \cot(\frac{b_m\pi}{p})+ \sum_{Y\subset X^G}
                      \csc^2(\frac{c_Y\pi}{p})\cdot (Y\cdot Y),
$$
where $Y\cdot Y$ denotes the self-intersection number of $Y$.

(2) (Weaker version) Denote by $\text{Sign}$ the signature of a space. Then
$$
p\cdot\text{Sign}(X/G)=\text{Sign}(X)
+\sum_{m\in X^G}\text{def}_m +\sum_{Y\subset X^G}\text{def}_Y,
$$
where the terms $\text{def}_m$ and $\text{def}_Y$ {\em(}called  {\em signature defect)} are given 
by the following formulae:
$$
\text{def}_m=\sum_{1\neq\lambda\in\C, \lambda^p=1}\frac{(1+\lambda^{a_m})(1+\lambda^{b_m})}
{(1-\lambda^{a_m})(1-\lambda^{b_m})}
$$
and 
$$
\text{def}_Y=\frac{p^2-1}{3}\cdot (Y\cdot Y).
$$
\end{prop}

It turns out that locally linear topological actions are largely determined by the associated data
$\Gamma=(\rho, X^G, \{\rho_x\})$, at least for the case where $G\equiv\Z_p$ and the action is pseudofree
(or free). First of all, there is the following realization result (for odd $p$) due to Edmonds and Ewing \cite{EE2}.

\begin{thm}
(Edmonds-Ewing \cite{EE2})
Let a closed, oriented, simply connected $4$-manifold $X$ be given, together with a representation
of $\Z_p$ on $H^2(X)$ of the form $H^2(X)=\Z[\Z_p]^r\oplus \Z^t$, preserving the cup product, and with
a candidate fixed point data for $t+2$ isolated fixed points satisfying the $G$-signature theorem and
an additional torsion condition (which vanishes for relatively small $p$). Then there is a locally linear
topological $\Z_p$-action on $X$ realizing the given data. 
\end{thm}

The classification up to topological conjugacy was obtained by Wilczynski \cite{Wil2} and 
Bauer-Wilczynski \cite{BW}, showing in particular that there are at most finitely many different conjugacy classes for a given data $\Gamma=(\rho, X^G,\{\rho_x\})$. Their work generalized earlier work of Hambleton-Kreck \cite{HK} on free actions.

Historically, the study of group actions has often focused on the investigation of certain manifolds
which carry a natural geometric structure, where the geometric structure admits a large group
of automorphisms. The most important example in this regard is the unit sphere of an Euclidean space, which inherits a linear structure from the ambient space. In the examples below, we shall examine 
some natural group actions in dimension $4$. 

\begin{exm}
(1) Consider $\SS^4$, the unit sphere in $\R^5$. Writing $\R^5=\C^2\times\R$, then every orientation-preserving linear $\Z_p$-action on $\SS^4$ is conjugate to one of the following: 
$$
\lambda\cdot (z_1,z_2,x)\mapsto (\lambda z_1, \lambda^q z_2,x), \mbox{ where }
\lambda^p=1, 0\leq q<p. 
$$
In the above model, the fixed point set consists of two isolated points, the north pole $(0,0,1)$ and the
south pole $(0,0,-1)$, when $q\neq 0$, and is given by the embedded $2$-sphere $z_1=0$ when $q=0$.
The local representations $\rho_x$ are given by $\lambda\cdot (z_1,z_2)\mapsto (\lambda z_1, \lambda^q z_2)$, $\lambda^p=1$, and since $H^2(\SS^4)=0$, the induced representation $\rho$ on $H^2(\SS^4)$ is
trivial.

Though we are mainly concerned with orientation-preserving actions in this article, we point out that
the only orientation-reversing linear action on $\SS^4$ is the free involution given by the antipodal map. 

(2) Consider the complex projective plane $\C\P^2$. The models for complex linear $\Z_p$-actions are 
$$
\lambda\cdot [z_0 : z_1 : z_2]\mapsto [z_0 :\lambda^a z_1:\lambda^b z_2], \mbox{ where }
\lambda^p=1, 0\leq a<b<p. 
$$
These actions are all orientation-preserving and have a trivial representation on $H^2(\C\P^2)=\Z$.
The fixed point set consists of three isolated points $[0 : 0 : 1], [0: 1: 0]$, and $[1 : 0 : 0]$ when 
$a\neq 0$, with $\rho_x$ given by $(-b,a-b)$, $(-a, b-a)$ and $(a,b)$. When $a=0$ (in which case 
one can assume $b=1$), the fixed point
set consists of one isolated point $[0 : 0 : 1]$ and an embedded $2$-sphere, the 
complex line $z_2=0$. Note that
when $p=2$, $a=0$ is the only possibility. 

The anticomplex linear action is the involution $[z_0 : z_1 : z_2]\mapsto [\bar{z}_0 :\bar{z}_1:\bar{z}_2]$,
with the fixed point set being $\R\P^2=\{[x_0 : x_1 : x_2]|x_i\in\R\}$. The representation $\rho$
on $H^2(\C\P^2)$ is given by multiplication by $-1$, and the action preserves the orientation.
\end{exm}

\begin{ques}
Is every $\Z_p$-action on $\SS^4$ or $\C\P^2$ conjugate to a linear action?
\end{ques}

In the topological category the answer is yes for pseudofree actions; the case of $\SS^4$ was due to
Kwasik and Schultz \cite{KS1}, and the case of $\C\P^2$ was due to Edmonds and Ewing \cite{EE1}
for the local representations (see also Hambleton-Lee \cite{HL}) and to Hambleton, Lee and Madsen \cite{HLM} for the classification (see also Wilczynski \cite{Wil1}). 

For non-pseudofree actions the answer is no. Giffen \cite{Gif} gave the first examples of $\Z_p$-actions
on $\SS^4$ having a knotted $\SS^2$ as the fixed point set. Hambleton, Lee and Madsen \cite{HLM} used Giffen's examples to produce nonlinear $\Z_p\times \Z_p$-actions on $\C\P^2$, but no examples of nonlinear $\Z_p$-actions on $\C\P^2$ are known. 

For free involutions on $\SS^4$, there is a long story and the answer is no in both topological 
and smooth categories. 
Fintushel and Stern \cite{FS} gave the first example of an involution on $\SS^4$ whose quotient 
is a fake smooth $\R\P^4$. Later Gompf \cite{Gom} showed that one of the earlier examples of fake
smooth $\R\P^4$'s due to Cappell and Shaneson \cite{CS} has universal covering diffeomorphic to $\SS^4$;
recently Akbulut \cite{A} has shown that all of them have universal covering diffeomorphic to $\SS^4$. 

In Section 2, we discuss several results related to Question 1.1 on symplectic actions. 

\begin{ques}
Given a general simply connected $4$-manifold $X$, how complex is the set of associated data
$\Gamma=(\rho, X^G,\{\rho_x\})$ of all $\Z_p$-actions on $X$?
\end{ques}

First, we note that for a $4$-manifold $X$ which is not one of those with a relatively simple 
intersection form $(H^2(X),\cup)$ (e.g., $\SS^4$, $\C\P^2$, etc.),  it is already a complicated 
algebraic problem to describe the representation $\rho$ induced by a $\Z_p$-action. 

\begin{exm}
A $K3$ surface is a simply connected complex surface with trivial canonical bundle. All $K3$ surfaces
are K\"{a}hler, and have the same diffeomorphism type as smooth $4$-manifolds \cite{BPV}. 
An example of a $K3$ surface is given by the Fermat quartic $X_0=\{z_0^4+z_1^4+z_2^4+z_3^4=0\}
\subset \C\P^3$. The intersection form of a $K3$ surface is 
$$
H^2(X) = H\oplus H\oplus H\oplus E_8(-1)\oplus E_8(-1),
$$
known as the $K3$ lattice. 

Finite groups of automorphisms of a $K3$ surface have been extensively studied, particularly for 
the case of symplectic automorphisms, largely due to work of Nikulin \cite{N} and Mukai \cite{Mu}
--- these are the ones which act trivially on the canonical bundle. Nikulin classified the Abelian groups
and showed that $\Z_p$ for $p=2,3,5,7$ are the only cyclic groups of a prime order. He also showed 
that the representation $\rho$ is uniquely determined up to conjugacy; this allows one to determine 
$\rho$ by examining certain specific examples, cf. \cite{Mo, GS}. For $p=2$, it is fairly easy to 
describe it ---
$\rho$ is the representation which fixes the $3$ copies of $H$ and switches the $2$ copies of $E_8(-1)$.
For $p\geq 3$, it becomes fairly complicated to describe them. One often does it by describing the  
invariant lattice $H^2(X)^{\Z_p}$ and its orthogonal complement $\Omega_p:=(H^2(X)^{\Z_p})^\perp$.
For example, for $p=3$, $H^2(X)^{\Z_p}=H\oplus H(3)\oplus H(3)\oplus A_2\oplus A_2$ and 
$\Omega_p=K_{12}(-2)$, the Coxeter-Todd lattice of rank $12$ with the bilinear form multiplied by 
$-2$.

Not much is known about the set of representations of $\Z_p$ on the $K3$ lattice which are induced 
by a pseudofree $\Z_p$-action. Besides the examples realized by a symplectic automorphism, a few 
more can be found in \cite{CK1}, Thm 1.8.  Compare also Question 3.5. 
\end{exm}

Toward Question 1.2, one often focused on the following two specific questions: homological rigidity
and boundedness of group actions. More precisely, homological rigidity concerns whether the 
representation $\rho$ is faithful, i.e., whether an action is trivial if the induced action on homology 
is trivial, while boundedness of actions asks whether a given manifold 
could admit group actions of arbitrarily large order.  We should point out that these rigidity 
questions were also motivated by the corresponding results in dimension
$2$ (assuming the genus of the Riemann surface is at least two), both of which were due to 
Hurwitz. 

In dimension 4, holomorphic actions tend to be rigid. $K3$ surfaces provide 
a classical example of homological
rigidity; every homologically trivial automorphism of a $K3$ surface is trivial
\cite{BPV}. Peters \cite{Peters} extended
homological rigidity to elliptic surfaces (with a few exceptions), including in particular all
simply connected elliptic surfaces. Concerning boundedness of actions, Hurwitz's theorem was extended
to algebraic surfaces of general type, where the optimal bound was attained by Xiao \cite{Xiao1, Xiao2}. 

On the other hand, the following result of Edmonds \cite{Ed} shows that topological actions are rather
flexible and occur in abundance. 

\begin{thm}
(Edmonds \cite{Ed})
Let a closed, simply connected $4$-manifold $X$ be given. For any prime number $p>3$, there exists
a locally linear, pseudofree and homologically trivial, topological $\Z_p$-action on $X$.  
\end{thm}
(However, for non-Abelian groups McCooey \cite{McC} established 
homological rigidity for locally linear topological actions. ) 

In Section 4, we will discuss some results toward rigidity for symplectic actions. 

\vspace{2mm}

Finally, we consider the following question which is only valid for smooth actions. 

\begin{ques}
Given a smooth $4$-manifold, how much is the smooth structure reflected by its finite groups of
(smooth) symmetries and vice versa? 
\end{ques}

In higher dimensions, this question was extensively studied in the case of homotopy spheres 
(cf. \cite{Sch}). In Section 3, we will discuss some results in this direction for the case of homotopy 
$K3$ surfaces (smooth manifolds homeomorphic to a $K3$ surface). 

\vspace{2mm}

The main technical tools for obtaining these results are Seiberg-Witten invariants and pseudoholomorphic
curves, which are suitably adapted for the purpose of group actions. The readers are referred to the survey article \cite{C1} for an overview and the individual papers for more details.

\section{Classificational results}
Consider a smooth action of a finite group $G$ on $\SS^4$ which has an isolated fixed point $x\in \SS^4$.
We are concerned with the question as whether the action is smoothly conjugate to a linear action. 
By taking out an invariant neighborhood of $x$, we obtain a smooth action of $G$ on the $4$-ball
$\B^4$ where the action is free and linear on the boundary. It is clear that the linearity of the action 
on $\SS^4$ is equivalent to the linearity of the action on $\B^4$. Such a group $G$ which acts freely
and linearly on $\SS^3=\partial \B^4$ has been classified, cf. e.g. \cite{Wolf}. In particular, its center is
nontrivial,  hence contains a cyclic subgroup of prime order. Applying the Smith theory \cite{Bre} to
this subgroup, it follows easily that the $G$-action on $\B^4$ is semifree and has a unique fixed point
$y\in\B^4$. In the complement of an invariant neighborhood of $y$, the action is free with quotient 
an $h$-cobordism of $\SS^3/G$, which, according to \cite{KS2}, is in fact an $s$-cobordism. Hence
our problem is now reduced to whether a smooth $s$-cobordism of elliptic $3$-manifolds $\SS^3/G$
is necessarily trivial provided that its universal covering is trivial. There are nontrivial topological
$s$-cobordisms of elliptic $3$-manifolds \cite{CS1, KS1}. Some potentially nontrivial smooth examples 
were given by Cappell and Shaneson \cite{CS2};  one of the Cappell-Shaneson examples in \cite{CS2}
was shown by Akbulut \cite{Ak} to have a trivial universal covering. 

In \cite{C2, C3} we showed that a symplectic $s$-cobordism of elliptic $3$-manifolds (with standard
structure near the boundary) is smoothly trivial. Using a theorem of Gromov \cite{Gr} which says that
a symplectic structure on $\B^4$ is standard if it is standard near the boundray, we can reformulate 
the results in \cite{C2, C3} as follows.

\begin{thm}
A smooth action of a finite group $G$ on $\B^4$, which is free and linear on the boundary and 
preserves the standard symplectic structure on $\B^4$, is conjugate to a linear 
action by a diffeomorphism equaling identity near the boundary. 
\end{thm}

We conjectured in \cite{C3} that a smooth $s$-cobordism of elliptic $3$-manifolds whose universal 
covering is trivial is symplectic, and suggested to attack this problem using near symplectic geometry
\cite{T1}. One can reformulate the conjecture in \cite{C3} in terms of group actions on $\SS^4$.

\begin{conj}
A smooth finite group action on $\SS^4$ with an isolated fixed point is smoothly conjugate to a linear
action. 
\end{conj}

Next, we give a corollary of Theorem 2.1, regarding group actions on $\C\P^2$.

\begin{thm}
Suppose a finite group $G$ acts smoothly on $\C\P^2$ which has an isolated fixed point and preserves 
a symplectic structure $\omega$. Then the $G$-action is smoothly conjugate to a linear action. 
\end{thm}

\begin{proof}
Fix a $G$-equivariant, $\omega$-compatible  almost complex structure $J$. The key point is to show 
that there is an embedded $J$-holomorphic $2$-sphere $C$ representing a generator of $H_2(\C\P^2)$,
which is invariant under the action of $G$. Assume the existence of $C$ momentarily. Then the action of
$G$ in an invariant neighborhood of $C$ is symplectomorphic to a linear action by the equivariant
symplectic neighborhood theorem. Now by a theorem of Gromov \cite{Gr}, the complement of this neighborhood is symplectomorphic to $\B^4$ with the standard structure. By Theorem 2.1, 
the $G$-action is smoothly linear, which implies that the $G$-action on $\C\P^2$ is smoothly linear.

It remains to show that such a $J$-holomorphic $2$-sphere exists, for which we need to 
exploit the structure of $G$. Note that $G$ acts freely and linearly on $\SS^3$ because the 
$G$-action on $\C\P^2$ has an isolated fixed point. Such a group falls into two distinct classes: 
(i) cyclic groups, (ii) non-Abelian groups with a center containing a $\Z_2$-subgroup, cf. \cite{Wolf}. 
Considering case (i), let $g\in G$ be a generator. Then by the Lefschetz fixed point theorem, 
there must be at least two distinct points, say $x,y\in\C\P^2$, that are fixed under $g$. By a theorem of 
Gromov \cite{Gr}, there is a unique embedded $J$-holomorphic $2$-sphere
$C$ passing through $x,y$ such that $C$ is a generator of $H_2(\C\P^2)$. The uniqueness of $C$ 
implies that $C$ must be invariant under $g$, and this finishes the case (i). For case (ii), let $\tau\in G$
be the involution contained in the center of $G$. Then Proposition 1.1(4) implies that the
fixed point set of $\tau$ contains a $2$-dimensional component $C$, which must be a $2$-sphere by
Proposition 1.1(3). Furthermore, $C$ is clearly $J$-holomorphic.  By Proposition 1.2(2), $C$ 
must be the only component in the fixed point set of $\tau$, and $C^2=1$. The latter implies that
$C$ is a generator of $H_2(\C\P^2)$. Finally, since $\tau$ is in the center of $G$, $C$ must be invariant under $G$. This finishes the case (ii). 
\end{proof}

In view of Theorem 2.1, one naturally asks

\begin{ques}
Let $\omega_0$ be the standard symplectic structure on $\R^4$. Suppose a finite group $G$ acts on 
$(\B^4,\omega_0)$ via symplectomorphisms, which are free and linear on the boundary of $\B^4$. 
Is the $G$-action conjugate to a linear action by a symplectomorphism equaling identity near the 
boundary?
\end{ques}

We have a partial result, which was obtained in 2004 (cf. \cite{C5}). 

\begin{thm}
(Chen \cite{C5})
Let $G$ be a cyclic or metacyclic finite group acting on $(\B^4,\omega_0)$
via symplectomorphisms which are linear near the boundary of $\B^4$. (We do not 
assume the action is free near the boundary.)
Then the action of $G$ is conjugate to a linear action by a symplectomorphism
of $(\B^4,\omega_0)$ which is identity near the boundary.
\end{thm}

\begin{proof}
The proof is an equivariant version of the proof in McDuff-Salamon \cite{McS} about 
the uniqueness of symplectic structures on the $4$-ball. 

First of all, we may identify $\R^4$ with $\C^2$ such that the
action of $G$ near $\partial\B^4$ is complex linear.

Consider the linear action of $G$ on $\C^2$. We claim that there
exists a decomposition $\C^2=\C\oplus\C$ such that the action of
$G$ on $\C^2$ can be extended to $\SS^2\times\SS^2$, which is obtained
via projectivizations of each factor in $\C\oplus\C$. The case when
$G$ is cyclic is clear, we simply take $\C^2=\C\oplus\C$ to be a
decomposition into the eigenspaces of $G$. Now suppose $G$ is metacyclic,
with a cyclic normal subgroup $H$ such that $G/H$ is cyclic. If
each element of $H$ has two distinct eigenvalues, then we take
$\C^2=\C\oplus\C$ to be the decomposition into eigenspaces of $H$.
Since $H$ is normal, any $g\in G$ either preserves
$\C\oplus\C$ or switches the two factors. It is clear that the
action of $G$ extends to an action on $\SS^2\times\SS^2$. Finally,
suppose the elements of $H$ have repeated eigenvalues. In this case $H$
lies in the center of $G$. We pick a $g\in G$ such that
$[g]\in G/H$ generates $G/H$. Then a decomposition $\C^2=\C\oplus\C$
into eigenspaces of $g$ will do.

With the preceding understood, we extend the $G$-action on $\B^4$ to
the rest of $\C^2$ by linearity, and fix a decomposition $\C^2=\C\oplus\C$
as described above. Then we compactify $\C\oplus\C$ via
projectivizations of each factor into $M=\SS^2\times\SS^2$. Note that
$M$ has the standard split symplectic form $\omega_0$ which has equal
areas on $\SS^2\times\{\infty\}$ and $\{\infty\}\times\SS^2$. Moreover,
the action of $G$ extends naturally to an action on $(M,\omega_0)$
which is linear near $\SS^2\times\{\infty\}\cup \{\infty\}\times\SS^2$.

Fix an $\omega_0$-compatible almost complex structure $J$ on $M$
such that $J$ is $G$-equivariant and is the product complex structure
near $\SS^2\times\{\infty\}\cup \{\infty\}\times\SS^2$. Then there are
two families of $J$-holomorphic embedded $2$-spheres representing the
classes of $\SS^2\times\{\infty\}$ and $\{\infty\}\times\SS^2$ respectively.
Call an element in the former a $A$-curve and an element in the latter a
$B$-curve. Note that under the action of $G$, a $A$-curve is either
sent to a $A$-curve or a $B$-curve.

Now consider $\SS^2\times\SS^2$, which is equipped with the corresponding
linear $G$-action, and is given with a compatible identification of
$\SS^2\times\{\infty\}$ and $\{\infty\}\times\SS^2$ to those in $M$.
Then there is a diffeomorphism $\psi:\SS^2\times\SS^2\rightarrow M$ defined
by $\psi(z,w)=\mbox{image }u_w\cap\mbox{image }v_z$, where $u_w$ is the
$A$-curve passing through $\{\infty\}\times\{w\}\in M$ and $v_z$ is the
$B$-curve passing through $\{z\}\times\{\infty\}\in M$. By way of construction,
$\psi$ is identity along the union of $\SS^2\times\{\infty\}$ and
$\{\infty\}\times\SS^2$. Moreover, note that for any $g\in G$,
either $g(u_w)=u_{gw}$ and $g(v_z)=v_{gz}$, or $g(u_w)=v_{gw}$ and
$g(v_z)=u_{gz}$, depending on whether $g$ preserves $\SS^2\times\{\infty\}$
and $\{\infty\}\times\SS^2$ or switches them. In any event, it
follows easily that $\psi$ is equivariant with respect to the linear
$G$-action on $\SS^2\times\SS^2$ and the $G$-action on $M$.

The diffeomorphism $\psi$ is further modified into a diffeomorphism
$\psi^\prime:\SS^2\times\SS^2\rightarrow M$ which is identity near
$\SS^2\times\{\infty\}\cup \{\infty\}\times\SS^2$ and equals
$\psi$ outside of a neighborhood of the union of $\SS^2\times\{\infty\}$
and $\{\infty\}\times\SS^2$. To define $\psi^\prime$, note that
in a neighborhood of $\SS^2\times\{\infty\}$, $\psi$ is given by
$(z,w)\mapsto (\phi(z,w),w)$ for a family of holomorphic functions
$w\mapsto \phi(z,w)$, $z\in\SS^2$, with $\phi(z,\infty)=z$, 
and similarly in a neighborhood
of $\{\infty\}\times\SS^2$, $\psi$ is given by $(z,w)\mapsto
(z,\phi^\prime(z,w))$ for a family of holomorphic functions
$z\mapsto\phi^\prime(z,w)$, $w\in\SS^2$, with $\phi^\prime(\infty,w)=w$. 
We fix a cutoff function
$\beta$ which equals $\infty$ near $\infty$ and equals $1$ outside a
neighborhood of $\infty$, and define $\rho(x)=\beta(|x|)x$ for $x$ in
a neighborhood of $\infty\in\SS^2$. With these understood, $\psi^\prime$
is defined by $(z,w)\mapsto (\phi(z,\rho(w)),w)$ in a neighborhood
of $\SS^2\times\{\infty\}$ and by $(z,w)\mapsto (z,\phi^\prime(\rho(z),w))$
in a neighborhood of $\{\infty\}\times\SS^2$. We claim that $\psi^\prime$
is also $G$-equivariant. To see this, note that the linear $G$-action on
$\SS^2\times\SS^2$ is given by either $g\cdot (z,w)=(az,bw)$ or
$g\cdot (z,w)=(aw,bz)$ for some $a,b\in\C$ such that $|a|=|b|=1$.
The fact that $\psi$ is equivariant and that the action of $G$ on
$M$ is linear near $\SS^2\times\{\infty\}\cup \{\infty\}\times\SS^2\subset M$
implies that $\phi(az,bw)=a\phi(z,w)$ in the former case and $\phi^\prime
(aw,bz)=b\phi(z,w)$ in the latter case. It is easy to check that
$\phi(az,\rho(bw))=a\phi(z,\rho(w))$ in the former case and
$\phi^\prime(\rho(aw),bz)=b\phi(z,\rho(w))$ in the latter case,
which implies that $\psi^\prime$ is $G$-equivariant.

Now apply Moser's argument to the following family of symplectic forms
on $\SS^2\times\SS^2$
$$
\omega_t=(1-t)\omega_0+t(\psi^\prime)^\ast\omega_0, \; 0\leq t\leq 1.
$$
Note that each $\omega_t$ is $G$-equivariant with respect to the linear
$G$-action on $\SS^2\times\SS^2$, and $\omega_t=\omega_0$ for all
$t\in [0,1]$ near $\SS^2\times\{\infty\}\cup \{\infty\}\times\SS^2$.
It follows easily that there exists a $1$-form $\alpha$ which obeys:
(1) $\frac{d}{dt}\omega_t=d\alpha$, (2) $\alpha$ vanishes near
$\SS^2\times\{\infty\}\cup \{\infty\}\times\SS^2$, and (3) $\alpha$
is $G$-equivariant. Let $X_t$ be the time-dependent vectorfield
on $\SS^2\times\SS^2$ which is defined by the equation
$\alpha+i(X_t)\omega_t=0$, then $X_t$ is clearly also $G$-equivariant,
and vanishes near $\SS^2\times\{\infty\}\cup \{\infty\}\times\SS^2$.
Let $\psi_1$ be the time-one map generated by $X_t$, then $\psi_1$ is identity
near $\SS^2\times\{\infty\}\cup \{\infty\}\times\SS^2$ and is $G$-equivariant
with respect to the linear $G$-action on $\SS^2\times\SS^2$. Moreover, 
$\psi_1^\ast ((\psi^\prime)^\ast\omega_0)=\omega_0$. 

Define $\Psi:\C^2\rightarrow\C^2$ to be the restriction of
$\psi^\prime\circ\psi_1$ to $\C\oplus\C\subset\SS^2\times\SS^2$
which is identity near infinity.
Then $\Psi^\ast\omega_0=\omega_0$ and $\Psi$ is $G$-equivariant
where $G$ acts on the domain $\C^2$ linearly and acts on the range
$\C^2$ by the natural extension of the action of $G$ on $\B^4$. Now apply 
the following equivariant version of Lemma 9.4.10 in McDuff-Salamon \cite{McS} to 
$f\equiv \Psi^{-1}|_{\C^2\setminus V}$, where $V\subset \B^4$ is a compact, 
$G$-invariant convex sub-domain such that $0\in int(V)$ and $G$ acts linearly 
on $\B^4\setminus V$:

\vspace{2mm}

{\it Let $G$ be a compact Lie group acting on $(\R^{2n},\omega_0)$ via
linear symplectomorphisms. Suppose $V\subset\R^{2n}$ is a $G$-invariant,
star-shaped compact set with $0\in int(V)$ and $f:\R^{2n}\setminus
V\rightarrow \R^{2n}$ is a $G$-equivariant symplectic embedding equaling
identity near infinity. Then, for every $G$-invariant open
neighborhood $W\subset\R^{2n}$ of $V$, there exists a $G$-equivariant
symplectomorphism $g:\R^{2n}\rightarrow\R^{2n}$ such that
$g|_{\R^{2n}\setminus W}=f$.
}
\vspace{2mm}

Let $g:\C^2\rightarrow \C^2$ be the $G$-equivariant 
symplectomorphism obtained from the lemma. 
Then $\Phi\equiv (\Psi\circ g)|_{\B^4}:\B^4\rightarrow \B^4$ is identity
near the boundary and satisfies $\Phi^\ast\omega_0=\omega_0$, and
is $G$-equivariant where the action of $G$ on the domain of $\Phi$
is linear while on the range of $\Phi$ it is the given action on $\B^4$.
\end{proof}

There are two immediate corollaries. As argued in Theorem 2.3, we obtain a
classification for certain finite subgroups of the symplectomorphism group of $\C\P^2$. 

\begin{thm}
(Chen \cite{C5})
Let $G\subset Symp(\C\P^2,\omega_0)$ be a finite subgroup which is
either cyclic or metacyclic with a nonempty fixed point set. Then $G$ is
conjugate in $Symp(\C\P^2,\omega_0)$ to a subgroup of $PU(3)$.
Here $\omega_0$ is the standard K\"{a}hler structure.
\end{thm}

With Example 3.5 in \cite{C1}, we obtain a classification of symplectic structures on weighted
$\C\P^2$'s, generalizing the corresponding theorem for $\C\P^2$ due to Gromov and
Taubes \cite{Gr,T}.

\begin{thm}
(Chen \cite{C5})
Two symplectic forms $\omega_1,\omega_2$ on a weighted projective plane
$\P(d_1,d_2,d_3)$ are cohomologous if and only if there exists a homologically trivial 
self-diffeomorphism $\psi$ of $\P(d_1,d_2,d_3)$ such that $\psi^\ast\omega_2=\omega_1$.
\end{thm}

Here $\P(d_1,d_2,d_3)$ is the orbifold $\SS^5/\SS^1$, where $\SS^1$ acts linearly with
weights $d_1,d_2,d_3$, i.e., $\lambda \cdot (z_1,z_2, z_3)= 
(\lambda^{d_1}z_1,\lambda^{d_2}z_2,\lambda^{d_3}z_3)$. (We assume $d_1,d_2,d_3$ are
relatively prime. )

\vspace{2mm}
Note that in all of the results discussed above concerning group actions on $\C\P^2$, 
there is always the assumption that the action has a fixed point. The only linear action 
on $\C\P^2$ which is pseudofree without
a fixed point is by the metacyclic group $\Gamma_{n,r}$, where 
$$
\Gamma_{n,r}=\{x,u\mid x^n=u^3=1,u^{-1}xu=x^r, r^2+r+1\equiv 0\pmod{n}\}.
$$
A linear model for the actions of $\Gamma_{n,r}$ on $\C\P^2$ is given below
where $\lambda^n=1$:
$$
x\cdot [z_0:z_1:z_2]=[z_0:\lambda^{-r}z_1:\lambda z_2] \mbox{ and }
u\cdot [z_0:z_1:z_2]=[z_2:z_0:z_1]. 
$$
It was shown in Wilczynski \cite{Wil1} that any pseudofree, locally linear topological action of
$\Gamma_{n,r}$ on $\C\P^2$ is conjugate to a linear action. 

\begin{ques}
(Chen \cite{C5})
Is every pseudofree symplectic $\Gamma_{n,r}$-action on $\C\P^2$ smoothly conjugate to a 
linear action?
\end{ques}

\section{Symmetries and exotic smooth structures}
A classical theorem in Riemannian geometry says that if a compact Lie group $G$ acts smoothly
and effectively on a compact closed $n$-dimensional manifold $M^n$, then the dimension of $G$ 
can not exceed $n(n+1)/2$, where if indeed $\dim G=n(n+1)/2$, $M^n$ must be diffeomorphic to 
$\SS^n$ or $\R\P^n$.  In dimensions greater than $6$, there exist exotic spheres, i.e., smooth manifolds
which are homeomorphic but not diffeomorphic to the standard sphere $\SS^n$. Thus the said theorem
in Riemannian geometry gives us a criterion for distinguishing the standard sphere from an exotic one
in terms of groups of symmetries --- the standard sphere has the largest group of symmetries. On
the other hand, there are exotic spheres whose symmetry groups are found to be very restricted. For
example, Lawson and Yau \cite{LY} showed that there are exotic spheres which do not even support 
actions of small groups such as $\SS^3$ or $SO(3)$. 

In Chen-Kwasik \cite{CK1,CK2}, we investigated these phenomena in dimension $4$, 
using homotopy $K3$ surfaces
as the testing ground. On the one hand, the finite symplectic automorphism groups of a $K3$ 
surface (called a  {\it symplectic $K3$ group}) have been classified, cf. e.g. Mukai \cite{Mu}, 
which provided us with a large list of finite groups that can act at least continuously on a
homotopy $K3$ surface. On the other hand, there are well-established methods, cf. e.g. \cite{FS1, T},
to construct and analyze exotic smooth structures on a homotopy $K3$ surface. 

The basic idea of our approach in \cite{CK1,CK2} can be summarized as follows. Let $X$ be a
homotopy $K3$ surface which admits a symplectic structure. Then deep work of Taubes \cite{T}
implies that the set of Seiberg-Witten basic classes of $X$ spans an isotropic sublattice 
$L_X$ of $H^2(X)$ (with respect to the cup product), so that its rank, denoted by $r_X$, must 
range from $0$ to $3$, where $3=\min (b_2^{+},b_2^{-})$. 
The rank $r_X$ of the lattice $L_X$ gives a rough measurement of the 
exoticness of the smooth structure of $X$, with $r_X=0$ being the least exotic and with $r_X=3$ 
being the most exotic. Furthermore, $r_X=0$ implies that $X$ has a trivial canonical bundle, and 
the standard $K3$ is the only known example to have $r_X=0$. (We will call $X$ ``standard"
if $r_X=0$.)

The link between the exoticness $r_X$ and symmetry groups of $X$ is made as follows. Suppose
a finite group $G$ acts smoothly on $X$. Then the lattice $L_X$ is invariant under the $G$-action,
and since $L_X$ is isotropic, $L_X\subset L_X^\perp$, so that there is an induced $G$-action on 
$L_X^\perp/L_X$. In other words, the $G$-representation $\rho$ on $H^2(X)$ breaks up into two
smaller pieces, one on $L_X$ and the other on $L_X^\perp/L_X$. If one further assumes that $G$
preserves a symplectic structure on $X$, then the pseudoholomorphic curve techniques developed in
\cite{CK} allow one to analyze the fixed point set structure $(X^G, \{\rho_x\})$,
and then through the Lefschetz fixed point theorem and the $G$-signature theorem, one obtains
information about the $G$-representation $\rho$ on $H^2(X)$. 

With the preceding understood, we investigated the following question in \cite{CK1}. Suppose we 
change the smooth structure of the standard $K3$. Is it possible to arrange it so that some of the
group actions on the standard $K3$ are no longer smoothable, and/or some of the groups which
can act smoothly on the standard $K3$ can no longer act smoothly on the exotic $K3$?

\begin{thm}
(Chen-Kwasik \cite{CK1})
There exist infinitely many symplectic homotopy $K3$ surfaces $X_\alpha$ with the exoticness 
$r_{X_\alpha}=3$, such that the following statements hold.
\begin{itemize}
\item [{(1)}] Let $g$ be a holomorphic automorphism of prime order $\geq 7$ of a $K3$ surface,
or an anti-holomorphic involution of a $K3$ surface whose fixed-point set contains a component 
of genus $\geq 2$. Then $g$ is not smoothable on $X_\alpha$. 
\item [{(2)}] Let $G$ be a finite group whose commutator $[G,G]$ contains a subgroup isomorphic 
to $(\Z_2)^4$ or $Q_8=\{i,j,k|i^2=j^2=k^2=-1, ij=k, jk=i,ki=j\}$, where in the case of $Q_8$ the 
elements of order $4$ in the subgroup are conjugate in $G$. Then there are no effective smooth $G$-actions on $X_\alpha$. 
\end{itemize}
\end{thm}

\begin{rem}
(1) The automorphism $g$ in Theorem 3.1(1) does exist \cite{N,N1,MO}, which 
provides the first examples of relatively
non-smoothable, locally linear topological actions on a $4$-manifold, i.e., examples of actions which are
smoothable with respect to one smooth structure but non-smoothable with respect to some other smooth
structures.  There were previously known examples of absolutely non-smoothable actions, for example,
there exists a locally linear topological action of order $5$ on $\C\P^2 \# \C\P^2$ which is not smoothable with respect to any smooth structure on $\C\P^2 \# \C\P^2$, see \cite{EE2, HL}. For more recent 
examples of absolutely non-smoothable actions, see e.g. \cite{LN1, LN2}. See also the recent survey
by Edmonds \cite{Ed2}.

(2) Among the $11$ maximal symplectic $K3$ groups (see \cite{Mu, Xiao} for a complete list), there
are $6$ of them satisfying the assumption in Theorem 3.1(2). They are 
$$
M_{20}, F_{384}, A_{4,4}, T_{192}, H_{192}, T_{48}. 
$$
On the other hand, the symplectic $K3$ group of smallest order which satisfies the assumption 
in Theorem 3.1(2) is the binary tetrahedral group $T_{24}$ of order $24$. 

(3) The $X_\alpha$'s in Theorem 3.1 were originally due to Fintushel and Stern \cite{FS1}. The crucial
fact we established in \cite{CK1} (see Lemma 4.2 therein) is that 
$$
L_{X_\alpha}^\perp/L_{X_\alpha}=E_8 (-1)\oplus E_8(-1).
$$
This implies that any smooth $G$-action on $X_\alpha$ induces a $G$-representation on the
$E_8$-lattice, which yields strong informations about the group $G$.
\end{rem}

We also analyzed in \cite{CK1} the fixed point set $X^G$ and local representations $\{\rho_x\}$ of
the symplectic $\Z_p$-actions on $X_\alpha$ which induce a nontrivial action on each of the two
$E_8$-lattices in $L_{X_\alpha}^\perp/L_{X_\alpha}=E_8 (-1)\oplus E_8(-1)$, cf. Theorem 1.8 therein. 
The following question remains open.

\begin{ques}
(Chen-Kwasik \cite{CK1})
Does there exist an exotic $K3$ surface which admits no symplectic $\Z_p$-actions for some $p\leq 19$?
\end{ques}

We remark that for every prime number $p\leq 19$, there is a holomorphic $\Z_p$-action on a $K3$
surface \cite{N,MO}. 

Now we review the main results in \cite{CK2}. Suppose we are given with a symplectic homotopy
$K3$ surface $X$ which admits a symplectic action of a ``large" $K3$ group (e.g., one of the $11$
maximal symplectic $K3$ groups). What can be said about the smooth structure of $X$ (e.g., in
terms of the exoticness $r_X$)? 

\begin{thm}
(Chen-Kwasik \cite{CK2})
Let $G$ be one of the following maximal symplectic $K3$ groups: 
$$
L_2(7), A_6, M_{20}, A_{4,4}, T_{192}, T_{48}
$$ 
and let $X$ be a symplectic homotopy $K3$ surface. If $X$ admits an 
effective symplectic $G$-action, then $X$ must be ``standard'' (i.e., $r_X=0$).
\end{thm}

The basic idea of the proof may be summarized as follows. 
Using the techniques developed in \cite{CK} and exploiting 
various features of the structure of $G$, one first determines the possible 
fixed point set of an arbitrary element $g\in G$, from which the trace $tr(g)$ of 
$g$ on $H^\ast(X;\R)$ can be computed using the Lefschetz fixed point theorem. 
This leads to an estimate on 
$$
\dim (H^\ast(X;\R))^G=\frac{1}{|G|}\sum_{g\in G}tr(g).
$$
On the other hand,  the following basic inequality
$$
\dim (L_X\otimes_\Z\R)^G\leq \min(b^{+}_2(X/G),b^{-}_2(X/G))
$$
plus the identity $\dim (H^\ast(X;\R))^G=2+b^{+}_2(X/G)+b^{-}_2(X/G)$ 
allows one to obtain information about $\dim (L_X\otimes_\Z\R)^G$ 
and $r_X=\text{rank }L_X$.

For an illustration we consider the case where $G$ is a non-Abelian simple
group. It is easily seen that in this case $b^{+}_2(X/G)=3$ and 
$r_X=\dim (L_X\otimes_\Z\R)^G$. The above basic inequality then becomes 
$$
r_X\leq \min(3, \dim (H^\ast(X;\R))^G-5).
$$
There are three non-Abelian simple $K3$ groups: $L_2(7), A_5$ and $A_6$.
For the case where $G=L_2(7)$ or $A_6$, one can show that 
$\dim (H^\ast(X;\R))^G=5$, so that $r_X=0$. For $G=A_5$, one can only show 
that $\dim (H^\ast(X;\R))^G\leq 8$, so that the basic inequality gives only 
$r_X\leq 3$, which does not yield any restriction on the exoticness $r_X$. 

\begin{ques}
(Chen-Kwasik \cite{CK2})
Suppose a symplectic homotopy $K3$ surface $X$ admits a symplectic $A_5$-action. Is $X$
necessarily ``standard" (i.e., $r_X=0$)?
\end{ques}

Theorem 3.2 gives a characterization of ``standard" symplectic homotopy $K3$ surfaces in terms of
finite symplectic symmetry groups. It naturally raises the question as what can be said about the finite
symplectic symmetry groups of a ``standard" symplectic homotopy $K3$ surface. The answer is given
in the following theorem.

\begin{thm}
(Chen-Kwasik \cite{CK2})
Let $X$ be a ``standard'' symplectic homotopy $K3$ surface (i.e. $r_X=0$) 
and let $G$ be a finite group acting on $X$ via symplectic symmetries. Then 
there exists a short exact sequence of finite groups 
$$
1\rightarrow G_0\rightarrow G\rightarrow G^0\rightarrow 1,
$$
where $G^0$ is cyclic and $G_0$ is a symplectic $K3$ group, such that
$G_0$ is characterized as the maximal subgroup of $G$ with the property
$b_2^{+}(X/G_0)=3$. Moreover, the action of $G_0$ on $X$ has the
same fixed-point set structure $(X^G,\{\rho_x\})$ of a symplectic holomorphic 
action of $G_0$ on a $K3$ surface.
\end{thm}

Theorem 3.3 raises the following questions.

\begin{ques}
(Chen-Kwasik \cite{CK2})
Are there any finite groups other than a $K3$ group which can act smoothly
or symplectically on a homotopy $K3$ surface?
\end{ques}

If a finite group $G$ other than a $K3$ group acts symplectically on a homotopy $K3$ surface $X$
such that $b_2^{+}(X/G)=3$, then $X$ must be an exotic $K3$. It would be interesting to see a
construction of such a homotopy $K3$ surface. Similarly,

\begin{ques}
(Chen-Kwasik \cite{CK2})
Are there any smooth or symplectic actions of a $K3$ group on a homotopy $K3$ surface 
which have a fixed-point set structure $(X^G,\{\rho_x\})$ different than a holomorphic action?
\end{ques}

According to Theorem 3.3, a symplectic $\Z_p$-action on a ``standard" symplectic homotopy $K3$
surface $X$ with $b_2^{+}(X/G)=3$ should have the same fixed-point set structure 
$(X^G,\{\rho_x\})$ of a symplectic 
automorphism of order $p$ of a $K3$ surface. A natural question is: how much will this determine the induced representation $\rho$ on the $K3$ lattice? In particular, does it determine 
$\rho$ completely? Note that in the holomorphic case, the affirmative answer of Nikulin in \cite{N}
relied heavily on the global Torelli theorem for $K3$ surfaces. (Note that over $\Q$, the
induced action is completely determined. )

\begin{ques}
Let a $\Z_p$-action on a homotopy $K3$-surface (locally linear topological, smooth, or symplectic)
be given which has the same fixed-point set structure $(X^G,\{\rho_x\})$ of an order $p$ symplectic automorphism of a $K3$ surface. What can be said about the induced representation $\rho$ 
on the $K3$ lattice?
\end{ques}

For $p=2$, we actually have an affirmative answer (however, for homeomorphism types up to
conjugacy the answer seems not that simple, cf. \cite{CKW}).

\begin{prop}
Let $g$ be a locally linear, orientation-preserving topological involution on a homotopy $K3$ surface 
$X$ which has $8$ isolated fixed points. Then the induced action of $g$ on the $K3$ lattice is the 
same as that of a symplectic automorphism of order $2$ of a $K3$ surface. 
\end{prop}

\begin{proof}
First of all, the Lefschetz fixed point theorem and the $G$-signature theorem imply that 
$\text{Sign} (X/\langle g\rangle)=-8$ and $b_2^{+}(X/\langle g\rangle)=3$. With this understood,
the proof is based on the following theorem of Nikulin \cite{N}. 

Denote by $L$ the $K3$ lattice. Suppose $G\subset O(L)$ is a finite subgroup of isometries of $L$.
Let $L^G$ denote the sublattice of $L$ fixed by $G$, and let $S_G\equiv (L^G)^{\perp}$ be
the orthogonal complement of $L^G$ in $L$. 

\begin{thm}
(Nikulin, Theorem 4.3 in \cite{N}) A finite subgroup $G\subset O(L)$ is realized as the induced action
of a finite symplectic automorphism group of a $K3$ surface if the following conditions hold:
\begin{itemize}
\item [{(a)}] $S_G$ is negative definite,
\item [{(b)}] $S_G$ does not have any elements with square $-2$,
\item [{(c)}] $\text{rank }S_G\leq 18$. 
\end{itemize}
\end{thm}

Now let $G\equiv\Z_2\subset O(L)$ be the subgroup of isometries of the $K3$ lattice induced by $g$. 
We shall
verify that the conditions (a), (b), (c) in Theorem 3.5 are satisfied. First of all, $S_G$ is negative definite
because $b_2^{+}(X/\langle g\rangle)=3=b_2^{+}(X)$. As for condition (c), 
$$
\text{rank }S_G=22-2b_2^{+}(X/\langle g\rangle)+\text{Sign} (X/\langle g\rangle)=8\leq 18.
$$

The verification of condition (b) is more involved. First of all, recall that by a theorem 
of Kwasik and Schultz 
(cf. \cite{KS}), the integral representation of $G\equiv\Z_2$ on $L$ can be expressed as a sum of copies
of the group ring $\Z[\Z_2]$, the trivial representation $\Z$, and the representation $\Z[\mu_2]$ of
cyclotomic type. By Proposition 1.1(3), there should be no summands of
cyclotomic type in the representation. With this said, there are elements
$x_1, \cdots, x_n, y_1,\cdots, y_m$ of $L$ such that the $x_i$'s are fixed by $g$ and 
$x_1, \cdots, x_n, y_1,\cdots, y_m$, $gy_1,\cdots, gy_m$ form a $\Z$-basis of $L$.  Now let $u\in S_G$
be any element. Note that $u$ obeys $gu=-u$. We write
$$
u=a_1 x_1+\cdots +a_nx_n +b_1y_1+\cdots +b_my_m +c_1 gy_1+\cdots +c_m gy_m.
$$
Then $gu=-u$ is equivalent to 
$$
2a_1 x_1+\cdots +2a_nx_n +(b_1+c_1)y_1+\cdots +(b_m+c_m)y_m +(b_1+c_1) 
gy_1+\cdots +(b_m+c_m) gy_m=0,
$$
which implies that $a_i=0$ for all $i$, and $b_j=-c_j$ for all $j$. Consequently,
$$
u=v-gv, \mbox{ where } v=b_1y_1+\cdots +b_my_m.
$$

Computing the square of $u$, we have 
$$
(u,u)=(v,v)-(v,gv)-(gv,v)+(gv,gv)=2(v,v)-2(v,gv).
$$
By Proposition 1.1(4), $(v,gv)=0 \pmod{2}$. On the other hand, $(v,v)=0\pmod{2}$ since $L$ is even. 
This shows that $(u,u)=0 \pmod{4}$. Particularly, $u$ does not have
square $-2$, and the condition (b) in Theorem 3.5 is verified. 
\end{proof}

One way to construct symplectic $\Z_p$-actions on the standard $K3$ surface is to fix a $C^\infty$-elliptic
fibration on the $K3$ surface and consider fibration-preserving $\Z_p$-actions. (It is a standard argument
to show that these actions are symplectic.)

\begin{ques}
Suppose a smooth $\Z_p$-action on the standard $K3$ surface preserves a $C^\infty$-elliptic fibration. 
Is it necessarily smoothly conjugate to a holomorphic action? What can be said about the induced
action on the $K3$ lattice?
\end{ques}

Finally, call a finite group $G$ (particularly a $K3$ group) ``small" if $G$ can act on a symplectic 
homotopy $K3$ surface $X$ with $r_X=3$ via symplectic symmetries. 

\begin{ques}
(Chen-Kwasik \cite{CK2})
What can be said about the set of ``small" groups?
\end{ques}

In \cite{CK2} we showed that $(\Z_2)^3$ is ``small".

\begin{thm}
(Chen-Kwasik \cite{CK2})
Let $G\equiv (\Z_2)^3$. There exists an infinite family of distinct symplectic homotopy 
$K3$ surfaces with maximal exoticness (i.e. $r_X=3$), such that each member
of the exotic $K3$'s admits an effective $G$-action via symplectic symmetries.
Moreover, the $G$-action is pseudofree and induces a trivial action on the
lattice $L_X$ of the Seiberg-Witten basic classes. 
\end{thm}

\section{Rigidity of group actions}
The main theme in \cite{C4} concerns boundedness of $\Z_p$-actions on a given
$4$-manifold. Of course, there are necessary conditions that have to be imposed.  By Edmonds'
theorem (cf. Theorem 1.4), it is necessary to consider at least smooth actions. On the other hand,
it is clear that if the manifold admits a smooth $\SS^1$-action, there is no chance for a bound on 
the order to exist. With this understood, the main question in \cite{C4} can be formulated as follows.

\begin{ques}
(Chen \cite{C4})
Let $X$ be a smooth $4$-manifold which admits no smooth $\SS^1$-actions. Does there 
exist a constant $C>0$ such that there are no nontrivial smooth $\Z_p$-actions on $X$ for any prime
number $p>C$? Moreover, suppose such a constant $C$ does exist,  what structures of $X$ does 
$C$ depend on?
\end{ques}

Before we state the theorems, it is helpful to recall the relevant results regarding the existence of
smooth $\SS^1$-actions on $4$-manifolds. First, with the resolution of the $3$-dimensional Poincar\'{e}
conjecture by Perelman, an old theorem of Fintushel \cite{F} may be strengthened to the following:
a simply connected $4$-manifold admitting a smooth $\SS^1$-action must be a connected sum of
copies of $\SS^4$, $\pm \C\P^2$, and $\SS^2\times\SS^2$. For the non-simply connected case, one 
has the following useful criterion due to Baldridge \cite{Bald} : 
Let $X$ be a $4$-manifold admitting a smooth $\SS^1$-action with nonempty fixed point set. 
Then $X$ has vanishing Seiberg-Witten invariant when $b_2^{+}>1$, and when $b_2^{+}=1$ and $X$ is symplectic, $X$ is diffeomorphic to a rational or ruled surface. For fixed-point free smooth
$\SS^1$-actions, formulas relating the Seiberg-Witten invariants of the $4$-manifold and the 
quotient $3$-orbifold were given in Baldridge \cite{B1,B2}. Note that if a $4$-manifold admits 
a fixed-point free $\SS^1$-action, the Euler characteristic and the signature of the manifold must vanish. 

With this understood, it is instructive to first look at the case of holomorphic actions, as
these are the primary source of
smooth actions on $4$-manifolds. Based on Xiao's generalization of Hurwitz's theorem in 
\cite{Xiao1, Xiao2} and the techniques developed by Ueno \cite{Ueno} and Peters \cite{Peters} 
regarding homological rigidity of holomorphic actions on elliptic surfaces, we obtained the 
following result. 

\begin{thm}
(Chen \cite{C4})
Let $X$ be a compact complex surface with Kodaira dimension $\kappa(X)\geq 0$. Suppose 
$X$ does not admit any holomorphic $\SS^1$-actions. Then there exists a constant $C>0$ such that
there are no nontrivial holomorphic $\Z_p$-actions of prime order on $X$ provided that $p>C$. 
Moreover, the constant $C$ depends linearly on the Betti numbers of $X$ and the order of the
torsion subgroup of $H_2(X)$, i.e., there exists a universal constant $c>0$ such that
$$
C=c(1+b_1+b_2+|Tor\; H_2|).
$$
\end{thm}

A natural question is that to what extent Theorem 4.1 holds also in the symplectic category. 
Let $(X,\omega)$ be a symplectic $4$-manifold such that $[\omega]\in H^2(X;\Q)$.  Denote by 
$N_\omega$ the smallest positive integer such that $[N_\omega\omega]$ is an integral class. We set
$$
C_\omega\equiv N_\omega c_1(K)\cdot [\omega],
$$ 
where $K$ is the canonical bundle of $(X,\omega)$. 

\begin{thm}
(Chen \cite{C4})
Let $(X,\omega)$ be a symplectic $4$-manifold with $b^{+}_2>1$ and
$[\omega]\in H^2(X;\Q)$, which does not satisfy the following condition: $X$ is minimal with vanishing
Euler characteristic and signature. Then there exists a constant $C>0$:
$$
C=c(1+b_1^2+b_2^2)C_\omega^2
$$
where $c>0$ is a universal constant, such that there are no nontrivial symplectic $\Z_p$-actions 
of prime order on $X$ provided that $p>C$. 
\end{thm}

A new phenomenon arising from Theorem 4.2 was that the constant $C$ may also depend on 
the smooth structure of the $4$-manifold, as $c_1(K)$ is a Seiberg-Witten basic class and
$C_\omega=N_\omega c_1(K)\cdot [\omega]$. Even though the following example does not say 
anything about the situation in Theorem 4.2, it does however suggest that the constant $C$ in 
Question 4.1 should be more than topological in nature. 

\begin{exm}
(Chen \cite{C4})

Let $X_0$ be the smooth rational elliptic surface given by the Weierstrass 
equation
$$
y^2z=x^3+v^5z^3.
$$
For any prime number $p\geq 5$, one can define an order-$p$ automorphism 
$g$ of $X_0$ as follows (cf. \cite{Zhang}):
$$
g: (x,y,z;v)\mapsto (\mu_p^{-5} x,y,\mu_p^{-15}z;\mu_p^6v),\;\; 
\mu_p\equiv\exp(2\pi i/p).
$$
Note that $g$ preserves the elliptic fibration and leaves exactly the two 
singular fibers (at $v=0$ and $v=\infty$) invariant.

Let $X_p$ be the symplectic $4$-manifold obtained from $X_0$ by 
performing repeated knot surgery (cf. \cite{FS1}) using the trefoil knot
on $p$ copies of regular fibers of the elliptic fibration which 
are invariant under the order-$p$ automorphism $g$. Then
it is clear that $X_p$ inherits a periodic symplectomorphism of order $p$.
On the other hand, $X_p$ is homeomorphic to $X_0$. (The key point is that
$X_p$ continues to be simply connected, as repeated knot surgery on parallel
copies of a regular fiber is equivalent to a single knot surgery using the connected 
sum of the knots, see \cite{C4} for more details.) 

To see that $X_p$ admits no smooth $\SS^1$-actions, notice that the canonical class 
of $X_p$ is given by the formula 
$$
c_1(K_{X_p})=(2p-1)\cdot [T], 
$$
where $[T]$ is the fiber class of the elliptic fibration which pairs 
positively with the symplectic form $\omega$ on $X_p$. 
Since $c_1(K_{X_p})\cdot [\omega]>0$ and $c_1(K_{X_p})^2=0$, 
$X_p$ is neither rational nor ruled, cf. \cite{Li}. 
By Baldridge's theorem \cite{Bald} or the strengthened version of the theorem of Fintushel 
in \cite{F}, $X_p$ does not admit any smooth $\SS^1$-actions.  

We thus obtained, for any prime number $p\geq 5$, a symplectic 
$4$-manifold $X_p$ homeomorphic to the rational elliptic surface, 
which admits no smooth $\SS^1$-actions but has 
a periodic symplectomorphism of order $p$. Observe that if $[\omega]$ is integral
(this can be arranged), the order $p$ satisfies 
$$
p\leq \frac{1}{2}(c_1(K_{X_p})\cdot [\omega]+1).
$$
\end{exm}

Note that in Example 4.1, the manifold $X_p$ has $b_2^{+}=1$. 
It is not clear, however, that with the condition $b_2^{+}>1$, whether 
the dependence of the constant $C$ on $\omega$ (as well as the 
assumption $[\omega]\in H^2(X;\Q)$) in Theorem 4.2 can be removed or not.
Note that  because of Theorem 4.1, the construction in Example 4.1 does not extend 
to the $b_2^{+}>1$ 
case by simply replacing the rational elliptic surface with some other elliptic surfaces 
with $b_2^{+}>1$.

\vspace{3mm}

The following is a symplectic analog of Xiao's theorem in \cite{Xiao1}.

\begin{ques}
(Chen \cite{C4})
Let $(X,\omega)$ be a minimal symplectic $4$-manifold of Kodaira
dimension $2$ (i.e., $c_1(K)\cdot [\omega]>0$ and $c_1(K)^2>0$, cf. \cite{Li}). Does there exists a 
universal constant $c>0$, such that
$$
|G|\leq c \cdot c_1(K)^2
$$
for any finite group $G$ of symplectomorphisms of $(X,\omega)$? 
\end{ques}

Note that the manifolds $X_p$ in Example 4.1 are minimal 
(cf. Gompf \cite{G} and Usher \cite{Usher}) with $c_1(K_{X_p})^2=0$,
so they are not of Kodaira dimension $2$ and do not give counterexamples.

\begin{ques}
(Chen \cite{C4})
For the general case of smooth $\Z_p$-actions, is the constant $C$ in Question 4.1
``describable" in terms of invariants of the manifold? In particular, what invariants of the smooth
structure (e.g.  SW basic classes, SW invariants, etc.) will enter the constant and how do they 
enter the constant? 
\end{ques}

The existence of smooth $\SS^1$-actions depends on the smooth structure in general (as shown 
by Example 4.1). However, by a theorem of Atiyah and Hirzebruch \cite{AH}, a spin 
$4$-manifold with non-zero signature does not admit smooth $\SS^1$-actions for any given smooth
structure. Example 4.1 naturally suggests the following question.

\begin{ques}
(Chen \cite{C4})
Let $X$ be a simply connected smoothable $4$-manifold with even intersection form and non-zero signature. Does there exist a constant $C>0$ depending on the topological type of $X$ only, such 
that for any prime number $p>C$, there are no $\Z_p$-actions on $X$ which are smooth with respect to
some smooth structure on $X$?
\end{ques}

We end this survey with two results concerning homological rigidity of symplectic $\Z_p$-actions from 
\cite{CK}. The following theorem extends the homological rigidity of holomorphic actions on a $K3$
surface to the symplectic category.

\begin{thm}
(Chen-Kwasik \cite{CK})
There are no homologically trivial, symplectic actions of a finite group on the standard $K3$
surface (with respect to any symplectic structure). 
\end{thm}

The following theorem partially extends Peters' results in \cite{Peters} to the symplectic category.

\begin{thm}
(Chen-Kwasik \cite{CK})
Let $X$ be a minimal symplectic $4$-manifold with $c_1^2=0$
and $b_2^{+}>1$, which admits a homologically trival (over $\Q$), pseudofree 
with $X^{G}\neq\emptyset$,
symplectic $\Z_p$-action for a prime $p>1$. Then the following conclusions hold.
\begin{itemize}
\item [{(a)}] The action is trivial if $p\neq 1\bmod 4$, $p\neq 1\bmod 6$.
In particular, for infinitely many primes $p$ the manifold
$X$ does not admit any such nontrivial $\Z_p$-actions.
\item [{(b)}] The action is trivial as long as there is a fixed point $x$ of 
local representation $\rho_x$ contained in $SL(2,\C)$. 
\end{itemize}
\end{thm}

We remark that while the homological rigidity question aims at giving some sort of understanding 
toward Question 1.2 regarding complexity of the set of associated data $\Gamma=(\rho,X^G,
\{\rho_x\})$ of all $\Z_p$-actions on a given $4$-manifold $X$, proofs of Theorems 4.3 and 4.4 
relied heavily on the knowledge of $(X^G,\{\rho_x\})$ which was
obtained through the pseudoholomorphic curve techniques developed in \cite{CK}. 

\vspace{2mm}

Regarding homological rigidity of smooth actions, the following question remains open, and is perhaps 
the most interesting one.

\begin{ques}
For any prime number $p>2$, are there any nontrivial but homologically trivial, smooth 
$\Z_p$-actions on a homotopy $K3$ surface?
\end{ques}

For the case of an involution, it is known that there are no locally linear, homologically trivial
topological actions on a homotopy $K3$ surface (see \cite{Ma, Rub}). On the other hand, Question 4.5
should be compared with Question 4.4, as for any $p>23$, a locally linear $\Z_p$-action 
on a homotopy $K3$ surface is automatically homologically trivial by Prop 1.1(1). 

\vspace{2mm}

{\bf Acknowledgements}: Many of the results surveyed in this article were joint work with 
S\l awomir Kwasik; I wish to thank him for the collaborations on this subject. Some of the ideas or
open questions discussed here might occur or were formed during or after a
conversation with him. The paper is an expanded version of the author's talk at the 16th G\"{o}kova Geometry/Topology Conference, Turkey, May 25-30, 2009. 
I wish to thank Selman Akbulut for encouraging me to write this survey.

\end{document}

%% file: goksty.tex
\def\E{\ifmmode{\mathbb E}\else{$\mathbb E$}\fi} 
\def\N{\ifmmode{\mathbb N}\else{$\mathbb N$}\fi} 
\def\R{\ifmmode{\mathbb R}\else{$\mathbb R$}\fi} 
\def\Q{\ifmmode{\mathbb Q}\else{$\mathbb Q$}\fi} 
\def\C{\ifmmode{\mathbb C}\else{$\mathbb C$}\fi} 
\def\H{\ifmmode{\mathbb H}\else{$\mathbb H$}\fi} 
\def\Z{\ifmmode{\mathbb Z}\else{$\mathbb Z$}\fi} 
\def\P{\ifmmode{\mathbb P}\else{$\mathbb P$}\fi} 
\def\T{\ifmmode{\mathbb T}\else{$\mathbb T$}\fi} 
\def\SS{\ifmmode{\mathbb S}\else{$\mathbb S$}\fi} 
\def\DD{\ifmmode{\mathbb D}\else{$\mathbb D$}\fi} 

\renewcommand{\a}{\alpha}
\renewcommand{\b}{\beta}
\renewcommand{\d}{\delta}
\newcommand{\D}{\Delta}
\newcommand{\e}{\varepsilon}
\newcommand{\g}{\gamma}
\newcommand{\G}{\Gamma}
\newcommand{\la}{\lambda}
\newcommand{\La}{\Lambda}
\newcommand{\n}{\nabla}
\newcommand{\var}{\varphi}
\newcommand{\s}{\sigma}
\newcommand{\Sig}{\Sigma}
\renewcommand{\t}{\tau}
\renewcommand{\th}{\theta}
\renewcommand{\O}{\Omega}
\renewcommand{\o}{\omega}
\newcommand{\z}{\zeta}
\newcommand{\B}{{\mathbb B}}

\newcommand{\ben}{\begin{enumerate}}
\newcommand{\een}{\end{enumerate}}
\newcommand{\be}{\begin{equation}}
\newcommand{\ee}{\end{equation}}
\newcommand{\bea}{\begin{eqnarray}}
\newcommand{\eea}{\end{eqnarray}}
\newcommand{\bc}{\begin{center}}
\newcommand{\ec}{\end{center}}

\newtheorem{thm}{Theorem}[section]
\newtheorem{cor}[thm]{Corollary}
\newtheorem{lem}[thm]{Lemma}
\newtheorem{prop}[thm]{Proposition}
\newtheorem{ax}{Axiom}
\newtheorem{conj}[thm]{Conjecture}

\theoremstyle{definition}
\newtheorem{defn}{Definition}[section]

\theoremstyle{remark}
\newtheorem{rem}{\rm\bfseries{Remark}}[section]
\newtheorem*{notation}{Notation}

\newtheorem{ques}{\rm\bfseries{Question}}[section]
\newtheorem{cons}[rem]{\rm\bfseries{Construction}}
\newtheorem{exm}[rem]{\rm\bfseries{Example}}
